\documentclass[11pt]{article}
\usepackage{amsthm,amssymb,amsmath,latexsym}
\usepackage{graphs}

\newtheorem{theorem}{Theorem}[section]
\newtheorem{lemma}[theorem]{Lemma}
\newtheorem{corol}[theorem]{Corollary}
\newtheorem{prop}[theorem]{Proposition}

\theoremstyle{definition}
\newtheorem{definition}[theorem]{Definition}
\newtheorem{example}[theorem]{Example}

\theoremstyle{remark}
\newtheorem{remark}[theorem]{Remark}

\numberwithin{equation}{section}

\begin{document}

\title{Orbit Equivalent Substitution Dynamical Systems and Complexity}
    
\author{ S.~Bezuglyi and O.~Karpel\\
Institute for Low Temperature Physics, \\
47 Lenin Avenue, 61103 Kharkov, Ukraine \\
(e-mail: bezuglyi@ilt.kharkov.ua, helen.karpel@gmail.com)}

%
\date{}
\maketitle
\begin{abstract}
For any primitive proper substitution $\sigma$, we give explicit constructions of  countably many pairwise non-isomorphic substitution dynamical systems $\{(X_{\zeta_n}, T_{\zeta_n})\}_{n=1}^{\infty}$ such that they all are (strong) orbit equivalent to $(X_{\sigma}, T_{\sigma})$. We show that  the complexity
of the substitution dynamical systems $\{(X_{\zeta_n}, T_{\zeta_n})\}$ is essentially different that prevents them from being isomorphic.
Given a primitive (not necessarily proper) substitution $\tau$, we find a stationary simple properly ordered Bratteli diagram with the least possible number of vertices such that the corresponding Bratteli-Vershik system is orbit equivalent to $(X_{\tau}, T_{\tau})$.

\end{abstract}

\maketitle

\section{Introduction}

The seminal paper \cite{GPS} answers, among other outstanding results,  the question of orbit equivalence of uniquely ergodic minimal homeomorphisms of a Cantor set. It was proved that two such minimal systems, $(X,T)$ and $(Y,S)$, are orbit equivalent if and only if the clopen values sets $S(\mu) =\{\mu(E) : E \ \mbox{clopen\ in\ X}\}$ and $S(\nu) =\{\nu(F) : F \ \mbox{clopen\ in\ Y}\}$ coincide where $\mu$ and $\nu$ are the unique invariant  measures  with respect to $T$ and $S$, respectively. It is well known now that Bratteli diagrams play an extremely important role in the study of homeomorphisms of Cantor sets because any minimal (and even aperiodic) homeomorphism of a Cantor set is conjugate  to the Vershik map acting on the path space of a Bratteli diagram \cite{GPS}, \cite{HPS}, \cite{M}. This realization turns out to be useful in many cases, in particular, for the study of substitution dynamical systems because the corresponding Bratteli diagrams are of the simplest form.
It was proved in  \cite{DHS} that the class of minimal substitution dynamical systems coincides with Bratteli-Vershik systems of stationary simple Bratteli diagrams. Later on, it was shown in \cite{BKM} that a similar  result is true for aperiodic dynamical systems.
These facts allow us to find easily the clopen values set $S(\mu)$ for a substitution dynamical system in terms of the matrix of substitution (see \cite{S.B.} and  subsection \ref{Bratteli diagrams}). In order to construct a minimal substitution dynamical system which is orbit equivalent to a given one, $(X_\sigma, T_\sigma)$, (in other words, a simple stationary Bratteli diagram $B_\sigma$) one has to find another stationary simple Bratteli diagram $B$ such that the clopen values set $S(\mu)$ is kept unchanged where $\mu$ is a unique $T_\sigma$-invariant measure.  Moreover, if one wants to have a substitution dynamical system which is strongly orbit equivalent to  $(X_\sigma, T_\sigma)$, then additionally the dimension group of the diagram $B_\sigma$ must be unchanged. Of course, we are not interesting in the case when powers of $\sigma$ are considered since it leads trivially to conjugate  substitution systems.

We focus here on  the study of orbit equivalence of minimal substitution  dynamical system because aperiodic non-minimal substitution systems were considered before in  \cite{S.B.O.K.}. We note that the simplest case when the invariant measure $\mu$ has rational $S(\mu)$ and $\lambda$ is an integer was studied in \cite{Yuasa}.

The main results of the present paper are as follows. Let $(X_\sigma, T_\sigma)$ be  a minimal substitution  dynamical system and let $B_\sigma$ be a stationary simple Bratteli diagram corresponding to $(X_\sigma, T_\sigma)$.
We give an explicit  construction of countably many substitutions $\{\zeta_n\}_{n=1}^\infty$ defined on the Bratteli diagrams $\{B_n\}_{n=1}^{\infty}$, obtained by telescoping of $B_\sigma$, such that the systems $\{(X_{\zeta_n}, T_{\zeta_n})\}_{n=1}^\infty$ are strong orbit equivalent to $(X_\sigma, T_\sigma)$ and pairwise non-isomorphic. In the other construction, we build pairwise non-isomorphic orbit equivalent minimal substitution dynamical systems by using alphabets of different cardinality.

In both constructions, we use the complexity function $n\mapsto p_\sigma(n)$ to distinguish non-isomorphic systems. Recall that the function $p_\sigma(n)$ counts the number of words of length $n$ in the infinite sequence invariant with respect to $\sigma$. In the first construction, the incidence matrices of built substitution systems are the powers of $A$. In the case of fixed alphabet, the complexity function can be made increasing by enlarging the length of substitution and an appropriate permutation of letters. Using this method, we produce a countable family of pairwise non-isomorphic strong orbit equivalent substitution systems.
 In the second construction, the complexity of the systems is growing by increasing
the number of letters in the alphabet. In other words, for a proper substitution $\sigma$ defined on the alphabet $\mathcal{A}$, we find countably many proper substitutions $\{\zeta_n\}_{n=1}^\infty$ on the alphabets $\mathcal{A}_n$ of different cardinality such that $(X_\sigma, T_\sigma)$ is orbit equivalent to $(X_{\zeta_n}, T_{\zeta_n})$,
but the set $\{(X_{\zeta_n}, T_{\zeta_n})\}_{n=1}^\infty$ consists of  pairwise
non-isomorphic substitution dynamical systems.

Given the Bratteli-Vershik system on a simple stationary diagram $(B, \leq)$, we find an orbit equivalent stationary Bratteli-Vershik system with the least possible number of vertices. This number is the degree of the algebraic integer $\lambda$, the Perron-Frobenius eigenvalue of the transpose $A$ to the incidence matrix of $B$.

\section{Preliminaries}

\subsection{Minimal Cantor systems} A minimal Cantor system is a pair $(X,T)$ where $X$ is a Cantor space and $T \colon X \rightarrow X$ is a minimal homeomorphism, i.e. for every $x \in X$ the set $Orb_T(x) = \{T^n (x) \; |\; n \in \mathbb{Z}\}$ is dense in $X$.

Given a minimal Cantor system $(X,T)$ and a clopen  $A\subset X$, let  $r_A(x) = \min\{n \geq 1 : T^n(x) \in A\}$ be a continuous integer-valued map defined on  $A$. Then $T_A(x) = T^{r_A}(x)$ is a homeomorphism of $ A$, and a Cantor minimal system $(A, T_A)$ is called  \textit{induced} from $(X,T)$.

There are several notions of equivalence for minimal Cantor systems:

\begin{definition}
Let $(X,T)$ and $(Y,S)$ be two minimal Cantor systems. Then

(1) $(X,T)$ and $(Y,S)$ are \textit{conjugate} (or \textit{isomorphic}) if there exists a homeomorphism $h \colon X \rightarrow Y$ such that $h \circ T = S \circ h$.

(2) $(X,T)$ and $(Y,S)$ are \textit{orbit equivalent} if there exists a homeomorphism $h \colon X \rightarrow Y$ such that $h(Orb_T(x)) = Orb_S(h(x))$ for every $x \in X$. In other words, there exist functions $n, m \colon X \rightarrow \mathbb{Z}$ such that for all $x \in X$, $h \circ T(x) = S^{n(x)} \circ h(x)$ and $h \circ T^{m(x)} = S \circ h(x)$. The functions $n,m$ are called orbit cocycles associated to $h$.

(3)  $(X,T)$ and $(Y,S)$ are \textit{strong orbit equivalent} if they are orbit equivalent and each of the corresponding orbit cocycles has at most one point of discontinuity.

(4) $(X,T)$ and $(Y,S)$ are \textit{Kakutani equivalent} if they have conjugate induced systems.

(5) $(X,T)$ and $(Y,S)$ are \textit{Kakutani orbit equivalent} if they have orbit equivalent induced systems.
\end{definition}

A Cantor system is called \textit{uniquely ergodic} if it has a unique invariant probability measure. For a full non-atomic Borel measure $\mu$ on a Cantor space $X$, define the \textit{clopen values set} $S(\mu) = \{\mu(U) : U \mbox{ clopen in } X\}$. Let $(X_1,T_1)$ and $(X_2,T_2)$ be two uniquely ergodic minimal Cantor systems and $\mu_1$ and $\mu_2$ be the unique probability invariant measures for $T_1$ and $T_2$, respectively. Among other results on orbit equivalence, it is proved in \cite{GPS} that  $(X_1,T_1)$ and $(X_2,T_2)$ are orbit equivalent if and only if $S(\mu_1) = S(\mu_2)$.

\subsection{Bratteli diagrams}\label{Bratteli diagrams}

\begin{definition}
A \textit{Bratteli diagram} is an infinite graph $B = (V,E)$ such that the vertex set $V = \bigcup_{i\geq t 0}V_{i}$ and the edge set $E = \bigcup_{i\geq 1}E_{i}$ are partitioned into disjoint subsets $V_{i}$ and $E_{i}$ such that

(i) $V_{0} = \{v_{0}\}$ is a single point;

(ii) $V_{i}$ and $E_{i}$ are finite sets;

(iii) there exist a range map $r$ and a source map $s$ from $E$ to $V$ such that $r(E_{i}) = V_{i}$, $s(E_{i}) = V_{i-1}$, and $s^{-1}(v)\neq 0$, $r^{-1}(v')\neq 0$ for all $v \in V$ and $v' \in V \setminus V_{0}$.
\end{definition}

The pair $(V_{i}, E_{i})$ or just $V_i$ is called the $i$-th level of the diagram $B$.
A sequence of edges $(e_{i} : e_{i} \in E_{i})$ such that $r(e_{i}) = s(e_{i + 1})$ is called a \textit{path}. We denote by $X_{B}$ the set of all infinite paths starting at the vertex $v_{0}$. This set  is endowed  with the standard topology  turning $X_{B}$ into a Cantor set.

Given a Bratteli diagram $B = (V,E)$, define a sequence of incidence matrices  $F_{n} = (f_{vw}^{(n)})$ of $B$:
$f_{vw}^{(n)} = |\{e \in E_{n+1} : r(e) = v, s(e) = w\}|$,
where $v \in V_{n+1}$ and $w \in V_{n}$. Here and thereafter $|V|$ denotes the cardinality of the set $V$.
A Bratteli diagram is called \textit{stationary} if $F_{n} = F_{1}$ for every $n \geq 2$.

A Bratteli diagram $B' = (V',E')$ is called the \textit{telescoping} of a Bratteli diagram $B = (V,E)$ to a sequence $0 = m_0 < m_1 < ...$ if $V_n' = V_{m_n}$ and $E_n'$ is the set of all paths from $V_{m_{n-1}}$ to $V_{m_n}$, i.e. $E_n' = E_{m_{n-1}}\circ...\circ E_{m_n} = \{(e_{m_{n-1}},...,e_{m_n}) : e_i \in E_i, r(e_i) = s(e_{i+1})\}$.

Observe that every vertex $v \in V$ is connected to $v_{0}$ by a finite path, and the set $E(v_0,v)$ of all such paths is finite.
A Bratteli diagram is called \textit{simple} if for any $n > 0$ there exists $m > n$ such that any two vertices $v \in V_n$ and $w \in V_m$ are connected by a finite path.


A Bratteli diagram $B =(V,E)$ is called {\it ordered} if every set
$r^{-1}(v)$, $v\in \bigcup_{n\ge 1} V_n$, is linearly ordered. Given an ordered
Bratteli diagram $(B,\leq) = (V, E, \leq)$, any two paths from $E(v_0,v)$ are
comparable with respect to the  lexicographical order \cite{HPS}. We call a finite
or infinite path $e=(e_i)$ {\it maximal (minimal)} if every $e_i$ is maximal
(minimal) amongst the edges from $r^{-1}(r(e_i))$. A simple
ordered Bratteli diagram $(B,\leq)$ is \textit{properly ordered} if there are  unique
 maximal and minimal infinite  paths. Any simple stationary Bratteli diagram
can be properly ordered.
A Bratteli diagram $B=(V,E, \leq)$ is called {\it stationary ordered} if it is stationary and the partial linear order on $E_n$ does not depend on $n$.

Let $(B,\leq) = (V,E,\leq)$ be a simple properly ordered stationary Bratteli diagram. Define a minimal homeomorphism $\phi_B \colon X_B \rightarrow X_B$ as follows. Let $\phi_B(x_{\max}) = x_{\min}$. If $x = (x_1, x_2,...) \neq x_{\max}$, let $k$ be the smallest number so that $x_k$ is not a maximal edge. Let $y_k$ be the successor of $x_k$ (hence $r(x_k) = r(y_k)$). Set $\phi_B(x) = (y_1,...,y_{k-1}, y_k, x_{k+1}, x_{k+2},...)$, where $(y_{1},...,y_{k-1})$ is the minimal path in $E(v_0, s(y_k))$. The resulting minimal Cantor system $(X_B, \phi_B)$ is called a \textit{Bratteli-Vershik system}. If $(B', \leq')$ is a telescoping of $(B, \leq)$ which preserves the lexicographical order then the Bratteli-Vershik systems $(X_B, \phi_B)$ and $(X_{B'}, \phi_{B'})$ are isomorphic.

\begin{definition}
Let $B = (V,E)$ be a Bratteli diagram. Two infinite paths $x = (x_{i})$ and $y = (y_{i})$ from $X_{B}$ are called \textit{tail equivalent} if there exists $i_{0}$ such that $x_{i} = y_{i}$ for all $i \geq i_{0}$. Denote by $\mathcal{R}$ the tail equivalence relation on $X_{B}$.
\end{definition}

A Bratteli diagram is simple if the tail equivalence relation $\mathcal R$ is minimal.
Denote $X_{w}^{(n)}(\overline{e}) := \{x = (x_{i}) \in X_{B} : x_{i} = e_{i}, i = 1,...,n\}$, where  $\overline{e} = (e_{1}, \ldots ,e_{n}) \in E(v_{0}, w)$, $n \geq 1$. A measure $\mu$ on $X_B$ is called $\mathcal{R}$-\textit{invariant} if for any two paths $\overline{e}$ and $\overline{e}'$ from $E(v_{0}, w)$ and any vertex $w$, one has $\mu(X_{w}^{(n)}(\overline{e})) = \mu(X_{w}^{(n)}(\overline{e}'))$. The measure invariant for a stationary Bratteli-Vershik system is $\mathcal{R}$-invariant.

In the paper, we will consider only simple stationary Bratteli diagrams. Let $A = F^T$ be the matrix transpose to the incidence matrix of a diagram $B$. Let $\lambda$ be a Perron-Frobenius eigenvalue of $A$ and let $x =  (x_1,...,x_K)^T$ be the corresponding positive eigenvector such that $\sum_{i=1}^K x_i = 1$. Suppose $B$ has no multiple edges between levels 0 and 1. Then the ergodic probability  measure $\mu$ defined by $\lambda$ and $x$ satisfies the relation:
$$
\mu(X_{i}^{(n)}(\overline{e})) = \frac{x_i}{\lambda^{n - 1}},
$$
where $i\in V_n$ and $\overline e$ is a finite path with $s(\overline e) =i$. Therefore, the clopen values set for $\mu$ has has the form:
$$
S(\mu) = \left\{\sum_{i = 1}^{K} k^{(n)}_{i}\frac{x_i}{\lambda^{n - 1}} : 0  \leq k^{(n)}_{i} \leq h^{(n)}_{i}; \; n = 1, 2, \ldots \right\},
$$
where $h_{v}^{(n)} = |E(v_{0}, v)|$, $v \in V_{n}$.
Let $H(x)$ be an additive subgroup of $\mathbb{R}$ generated by $x_1,...,x_n$. Since the Bratteli-Vershik system is minimal and $\mu$ is a unique invariant measure, $\mu$ is good and $S(\mu) = \left(\bigcup_{N=0}^\infty \frac{1}{\lambda^N} H(x)\right) \cap [0,1]$ (see~\cite{Akin2, S.B.O.K.}). It is easy to see that $\lambda H(x) \subset H(x)$ and $\lambda^m \in H(x)$ for any $m \in \mathbb{N}$ (see \cite{S.B.O.K.}).

\subsection{Substitution dynamical systems} Let $\mathcal{A} = \{a_1,...,a_s\}$ be a finite alphabet. Let $\mathcal{A}^{*}$ be the collection of finite non-empty words over $\mathcal{A}$. Denote by $\Omega = \mathcal{A}^\mathbb{Z}$ the set of all two-sided infinite sequences on $\mathcal{A}$. A substitution $\sigma$ is a map $\sigma \colon \mathcal{A} \rightarrow \mathcal{A}^{*}$. It extends to maps $\sigma \colon \mathcal{A}^{*} \rightarrow \mathcal{A}^{*}$ and $\sigma \colon \Omega \rightarrow \Omega$ by concatenation.
Denote by $T$ the shift on $\Omega$: $T(...x_{-1}.x_0 x_1 ...) = ...x_{-1}x_0.x_1 ...$.

Let $A_{\sigma} = (a_{ij})_{i,j = 1}^s$ be the incidence matrix associated to $\sigma$  where $a_{ij}$ is the number of occurrences of $a_i$ in $\sigma(a_j)$. Clearly,  $A_{\sigma^n} = (A_\sigma)^n$ for every $n \geq 0$.
A substitution $\sigma$ is called \textit{primitive} if there is $n$ such that for each $a_i,a_j \in \mathcal{A}$, $a_j$ appears in $\sigma^n (a_i)$. Note that $\sigma$ is primitive if and only if $A_\sigma$ is a primitive matrix. If it happens that $|\sigma(a)| = q$ for any $a\in \mathcal A$, then the substitution $\sigma$ is called of constant length $q$.
For $x \in \Omega$, let $L_n(x)$ be the set of all words of length $n$ occurring in $x$. Set $L(x) = \bigcup_{n \in \mathbb{N}} L_n(x)$. The language of $\sigma$ is the set $L_\sigma$ of all finite words occurring in $\sigma^n(a)$ for some $n \geq 0,\; a \in \mathcal{A}$. Set $X_\sigma = \{x \in \Omega : L(x) \subset L_\sigma\}$.

Throughout this paper we will consider only primitive substitutions $\sigma$ such that $X_\sigma$ is a Cantor set.
The dynamical system $(X_\sigma, T_\sigma)$, where $T_\sigma$ is the restriction of $T$ to the $T$-invariant set $X_\sigma$, is called \textit{the substitution dynamical system} associated to $\sigma$. It is well known (see~\cite{Q}) that every primitive substitution generates a minimal and uniquely ergodic dynamical system.

The following statements can be found in~\cite{Q}. For every integer $p > 0$ the substitution $\sigma^p$ defines the same language as $\sigma$, hence the systems $(X_\sigma, T_\sigma)$ and $(X_{\sigma^p}, T_{\sigma^p})$ are isomorphic. Substituting $\sigma^p$ for $\sigma$ if needed, we can assume that there exist two letters $r, l \in \mathcal{A}$ such that $r$ is the last letter of $\sigma(r)$, $l$ is the first letter of $\sigma(l)$ and $rl \in L_\sigma$.
The sequence $\omega  = \lim_{n \rightarrow \infty} \sigma^n(r.l) \in X_\sigma$ is a fixed point of $\sigma$ (that is $\sigma(\omega) = \omega$) and  $\omega_{-1} = r$, $\omega_{0} = l$.  Then $X_\sigma = \overline{Orb_T(\omega)}$.

 The \textit{complexity} of $u \in \Omega$ is the function $p_u(n)$ which associates to each integer $n \geq 1$ the cardinality of $L_n(u)$. It is easy to see that

\begin{equation}\label{fund_equal}
p_u(k+1) - p_u(k) = \sum_{w \in L_k(u)}(\mbox{Card}\;\{a \in \mathcal{A}: wa \in L_{k+1}(u)\} - 1).
\end{equation}

The sequence $u$ is called \textit{minimal} if every word occurring in $u$ occurs in an infinite number of places with bounded gaps. A fixed point of a primitive substitution is always minimal (see~\cite{F, Q}). Let $X_u$ be the set of all sequences $x \in \Omega$ such that $L_n(x) = L_n(u)$ for every $n \in \mathbb{N}$. For a primitive substitution $\sigma$ with the fixed point $u$ we have $X_\sigma = X_u = \overline{Orb_T(u)}$.
Hence $p_x(n) = p_u(n)$ for every $n$ and every $x \in X_u$. Sometimes we will denote $p_u$ by $p_\sigma$ to stress that the complexity function is defined by $\sigma$.

The following results can be found in~\cite{C, F, Q}.
\begin{theorem}\label{complexity_function}
(1) If the symbolic systems $(X_u, T)$ and $(X_v, T)$ associated to minimal sequences $u$ and $v$ are topologically conjugate, then there exists a constant $c$ such that, for all $n > c$,
$$
p_u(n - c) \leq p_v(n) \leq p_u(n + c).
$$
Hence a relation $p_u(n) \leq a n^k + \bar{o}(n^k)$ when $n \rightarrow \infty$ is preserved by conjugacy (isomorphism).

(2) Let $\zeta$ be a primitive substitution and  $p$ the complexity function of a fixed sequence $u = \zeta(u)$. Then, there exists a constant $C > 0$ such that $p(n) \leq Cn$ for every $n \geq 1$.

(3) Let $u \in \Omega$ and $p$ be the complexity function of $u$. Suppose that there exists $a > 0$ such that $p(n) \leq an$ for all $n \geq n_0$. Then
$$
p(n+1) - p(n) \leq K s a^3
$$
for all $n \geq n_0$, where $K$ does not depend on $u$.
\end{theorem}

\begin{definition}
A substitution $\sigma$ on an alphabet $\mathcal{A}$ is called \textit{proper} if there exists an integer $n > 0$ and two letters $a,b \in \mathcal{A}$ such that for every $c \in \mathcal{A}$, $a$ is the first letter and $b$ is the last letter of $\sigma^n(c)$.
\end{definition}

For every primitive substitution $\zeta$, there exists a proper substitution $\sigma$ such that the substitution systems $(X_{\zeta}, T_\zeta)$ and $(X_\sigma, T_\sigma)$ are isomorphic. The substitution $\sigma$ is built using the method of return words (see~\cite{DHS}).
The following theorem establishes the link between incidence matrices of $\zeta$ and $\sigma$:
\begin{theorem}\cite{Yuasa} \label{return_words}
Let $\zeta$ be a non-proper primitive substitution and let $\lambda$ be the Perron-Frobenius eigenvalue of its incidence matrix $A_\zeta$. Let $\sigma$ be the corresponding proper substitution built by means of return words. Then the Perron-Frobenius eigenvalue of $A_\sigma$ is $\lambda^k$ for some $k \in \mathbb{N}$.
\end{theorem}

Stationary Bratteli diagrams are naturally related to substitution dynamical
systems (primitive substitutions are considered in \cite{DHS}, \cite{Forrest},
and the non-primitive case is studied in \cite{BKM}). More precisely, let $(B,\leq)
  = (V,E, \leq)$ be a stationary ordered Bratteli diagram with no multiple edges
 between levels 0 and 1. Choose a stationary labeling of $V_n$ by an alphabet
$\mathcal A$: $V_n = \{ v_n(a) : a \in \mathcal A\},\ n >0$. For $a\in
\mathcal A$ consider the ordered set $(e_1,...,e_s)$ of edges that range at
$v_n(a), \ n\geq 2$. Let $(a_1,...,a_s)$ be the corresponding ordered set of the
 labels of the sources of these edges. The map $a \mapsto a_1\cdots a_s$
from $\mathcal A$ to $\mathcal A^*$ does not depend on $n$ and determines
a substitution called the substitution read on $(B,\leq)$.
Conversely, for any substitution dynamical system we can build the corresponding ordered stationary Bratteli diagram.
The following theorem, proved in~\cite{DHS}, shows the link between simple Bratteli diagrams and primitive substitution dynamical systems.

\begin{theorem}
Let $(B,\leq)$ be a stationary, properly ordered Bratteli diagram with only simple edges between the top vertex and the first level. Let $\sigma$ be the substitution read on $(B,\leq)$.

(i) If $\sigma$ is aperiodic, then the Bratteli-Vershik system $(X_B, \phi_B)$ is isomorphic to the substitution dynamical system $(X_\sigma, T_\sigma)$.

(ii) If $\sigma$ is periodic, the Bratteli-Vershik system $(X_B, \phi_B)$ is isomorphic to a stationary odometer.

\end{theorem}

In~\cite{Yuasa}, the following result was proved:

\begin{theorem}\label{oeodometer}
Let $\sigma$ be a primitive substitution whose incidence matrix has natural Perron-Frobenius eigenvalue. Then $(X_\sigma, T_\sigma)$ is orbit equivalent to a stationary odometer system.
\end{theorem}

We will need the next result  (see~\cite{DHS, Forrest}):
\begin{theorem}\label{mult_edges_thm}
Let $(B,\leq)$ be a stationary properly ordered Bratteli diagram. Then there exists a stationary properly ordered Bratteli diagram $(B', \leq')$ such that $B'$ has no multiple edges between levels 0 and 1 and the systems $(X_B, \phi_B)$, $(X_{B'}, \phi_{B'})$ are isomorphic.
\end{theorem}

\section{Orbit equivalence class for a primitive substitution}

Given a primitive proper substitution $\sigma$, we build countably many pairwise non-isomorphic substitution dynamical systems $\{(X_{\zeta_n}, T_{\zeta_n})\}_{n=1}^{\infty}$ in the orbit equivalence class of $(X_{\sigma}, T_{\sigma})$. Two essentially different constructions are elaborated. In the first one, we obtain countably many strong orbit equivalent substitution systems defined on the same alphabet. The second construction produces countably many orbit equivalent substitution systems with increasing cardinality of alphabets.
Finally, given a primitive (not necessarily proper) substitution $\tau$, we find a stationary simple properly ordered Bratteli diagram with the least possible number of vertices such that the corresponding Bratteli-Vershik system is orbit equivalent to $(X_{\tau}, T_{\tau})$.


\begin{theorem}
Let $(B, \leq)$ be a stationary properly ordered simple Bratteli diagram. Let $\sigma$ be the substitution read on $(B, \leq)$. Then there exist countably many telescopings $B_n$ of $B$ with proper orders $\leq_n$ and corresponding substitutions $\zeta_n$ read on $B_n$ such that the substitution dynamical
systems $\{(X_{\zeta_n}, T_{\zeta_n})\}_{n=1}^\infty$ are pairwise non-isomorphic and strong orbit equivalent to $(X_\sigma, T_\sigma)$.
\end{theorem}

\begin{proof}
Let $\mathcal{A} = \{a_1,...,a_s\}$ be the alphabet for $\sigma$.
Fix a number  $l \in \mathbb{N}$. Let $\{\omega_r\}_{r=1}^{s^l}$ denote the set of all possible words of length $l$ over the alphabet $\mathcal{A}$. Take arbitrary $N \in \mathbb N$ ($N$ will be chosen below) and consider the telescoping $B_N$ of $B$ with incidence matrix $A^N$. In our construction, we will define a proper substitution $\zeta = \zeta(l)$ that is read on the Bratteli diagram $B_N$ whose incidence matrix is $A_\zeta = A^N$. This means that the number of occurrences  of any letter $a_i$ in $\zeta(a_j)$ is known but we are free to choose any order of letters in the word $\zeta(a_j)$. In other words, we will change the lexicographical order $\leq_{\mbox lex}$, that obviously determines $\sigma^N$, in order to define $\zeta$. We take $N$ sufficiently large to guarantee that $\zeta$ satisfies the following conditions:

(1) for all $1 \leq j \leq s$ the word $\zeta(a_j)$ starts with the word $a_1 a_j$ and ends with the letter $a_1$;

(2) the word $\zeta(a_1)$ contains as subwords all words $\{\omega_i a_j\}$ for $1 \leq i \leq s^l$ and $1 \leq j \leq s$.

Obviously, it follows that  $\zeta$ is a proper substitution and the two substitution dynamical systems, $(X_{\zeta}, T_\zeta)$ and $(X_\sigma, T_\sigma)$, are strongly orbit equivalent  because the dimension groups associated to these minimal Cantor systems (that is to the diagrams $B$ and $B_N$) are order isomorphic by a map preserving the distinguished order unit (see~\cite{GPS}).

We need to show that for an appropriate choice of $l$ the substitution  $\zeta = \zeta(l)$ is such that the  systems  $(X_{\zeta(l)}, T_{\zeta(l)})$ and $(X_\sigma,T_\sigma)$  are not isomorphic. We see that  $\zeta^{\infty}(a_1.a_1) = \lim_{n\rightarrow \infty} \zeta^n(a_1.a_1)$ is a fixed point. The parameter $l$ should be  now chosen in such a way that  the complexity function $p_\zeta$ of  $\zeta^{\infty}(a_1.a_1)$  grows essentially faster then the complexity function associated to $(X_\sigma,T_\sigma)$.

To clarify the idea of the proof, we first prove the theorem in the case when $\sigma$ is a substitution of constant length $q$. Then $\zeta$ is a substitution of length $q^N$.

By definition of $\zeta$, we have $p_\zeta(1) = s$, $p_\zeta(2) = s^2$ and $p_\zeta(k) = s^k$ for $1 \leq k \leq l + 1$.
The word $\zeta^2(a_1)$ contains the words $\zeta(\omega_i) \zeta(a_j)$ for $1 \leq i \leq s^l$ and $1 \leq j \leq s$.
Recall that $\zeta(a_j)$ starts with $a_1 a_j$. Thus, $\zeta^2(a_1)$ contains $s^l$
different words $\{\zeta(\omega_i) a_1\}_{i = 1}^{s^l}$ and each word can be followed
 by any letter from $\mathcal{A}$. Since $lq+1 = |\zeta(\omega_i)| + |a_1|$, we
obtain $p_\zeta(lq+2) - p_\zeta(lq+1) \geq s^l(s-1)$ by (\ref{fund_equal}).

Consider $\zeta^3(a_1)$. Then apply the previous arguments with $\zeta(a_i)$ instead of $a_i$. Since $|\zeta^2(\omega_i)| + |\zeta(a_1)| + |a_1| = lq^2+ q + 1$, we get $p_\zeta(lq^2 + q + 2) - p_\zeta(lq^2+ q +1)\geq s^l(s-1)$. Thus, we conclude by induction that  for all $m \in \mathbb{N}$
 $$p_\zeta(lq^m + \sum_{i=0}^{m-1}q^i + 1) - p_\zeta(lq^m + \sum_{i=0}^{m-1}q^i)\geq s^l(s-1).
 $$
 Taking $l$ large enough, we can make the difference $p_\zeta(k+1) - p_\zeta(k)$ arbitrary large for infinite number of values of $k$.
 Now if we assumed that the substitution systems $(X_{\zeta}, T_{\zeta})$ and $(X_\sigma,T_\sigma)$  are isomorphic, then we would have the following relation that follows from  Theorem~\ref{complexity_function}: (i) there exists $C > 0$ such that $p_\sigma(n) \leq Cn$ for all $n \geq 1$ and  $p_\zeta(n) \leq (C+1)n$ for sufficiently large $n$; (ii) $p_\zeta(n+1) - p_\zeta(n) \leq Ks(C+1)^3$ for all sufficiently large $n$, where $K$ is a ``universal'' constant. Clearly, these statements contradict to the proved above fact that the values of $p_\zeta(n+1) - p_\zeta(n)$ are unbounded.

To prove the theorem in the general case, denote $q_m = |\zeta^m(a_1)|$ for $m \in \mathbb{N}$. Let $D_m = \max\limits_{1 \leq i \leq s}|\zeta^m(a_i)|$ and $d_m = \min\limits_{1 \leq i \leq s}|\zeta^m(a_i)|$.
Since $|\omega_i| = l$, we have $l d_m \leq |\zeta^m(\omega_i)| \leq l D_m$ for any $1 \leq i \leq s^l$ and $m \geq 1$.
 The matrix $A_\zeta$ is strictly positive, hence there exist positive constants $M_1$, $M_2$ such that for every $m \geq 1$ we have $M_1 \lambda^m \leq d_m \leq D_m \leq M_2 \lambda^m$, where $\lambda$ is the Perron-Frobenius eigenvalue of $A_\zeta$ (see~\cite{Q}). Hence $\frac{D_m}{d_m} \leq \frac{M_2}{M_1}$ for all $m \geq 1$.
For $r > 0$, set $l = l(r) = \left(\left[\frac{M_2}{M_1}\right] + 1\right) r$.
Then $l \geq r \frac{M_2}{M_1} \geq r \frac{D_m}{d_m}$ and $l d_m \geq r D_m$ for all $m \geq 1$. Thus, $\{\zeta^m(\omega_i)\}_{i=1}^{s^l}$ contains at least $s^r$ different suffices of length $D_m r$ and each suffix can be followed by any word from $\{\zeta^m(a_j)\}_{j=1}^s$. By the same argument as in the case of substitution of constant length, we conclude that
 $$p_\zeta(D_m r + \sum_{i=0}^{m-1}q_i + 1) - p_\zeta(D_m r +  \sum_{i=0}^{m-1}q_i) \geq s^r(s-1).
$$
By (\ref{fund_equal}), we obtain the needed result.

Thus, $p_\sigma(n) \leq Cn$ for all $n \geq 1$ but for any $N \in \mathbb{N}$ there exists $M > N$ such that $p_\zeta(M) > (C+1)M$.
Set $\zeta_1 = \zeta$. There exists $C_1 > 0$ such that $p_{\zeta_1}(n) \leq C_1 n$ for every $n \geq 1$. By the same method as above, we construct $\zeta_{m+1}$ using $\zeta_m$ for $m \in \mathbb{N}$. We obtain $p_{\zeta_{m}}(n) \leq C_{m} n$ for $n \geq 1$ and for any $N \in \mathbb{N}$ there exists $M > N$ such that $p_{\zeta_{m+1}}(M) > (C_m+1)M$.
Hence, by Theorem~\ref{complexity_function}, we obtain countably many pairwise non-isomorphic substitution dynamical systems $\{(X_{\zeta_m}, T_{\zeta_m})\}_{m=1}^\infty$ in the strong orbit equivalence class of $(X_\sigma, T_\sigma)$.
\end{proof}

In contrast to the first construction where the cardinality of the alphabet was fixed and this led us to a class of strongly orbit equivalent substitution systems, we will consider now a class of substitution dynamical systems defined on the alphabets of variable cardinality.

For a primitive matrix $A \in Mat(\mathbb{N},s)$ with the Perron-Frobenius eigenvalue $\lambda$ and the corresponding normalized eigenvector $x$, denote by $S(A) = \bigcup_{N=0}^\infty \frac{1}{\lambda^N} H(x)$.

\begin{lemma}\label{matrixbuild}
Let $A \in Mat(\mathbb{N},s)$ be a primitive matrix. Then there exist primitive matrices $\{A_n\}_{n=1}^{\infty}$, where $A_n \in Mat(\mathbb{N},s+n)$ such that $S(A) = S(A_n)$ for all $n \in \mathbb{N}$.
\end{lemma}

\begin{proof}  Let $\lambda$ be the Perron-Frobenius eigenvalue of $A$ and $x = (x_1,...,x_s)^T$ the corresponding normalized eigenvector. We construct an $(s+1) \times (s+1)$ matrix $A_1$ such that $A_1$ satisfies the condition of the lemma, and $y = (x_1,...,x_s, \lambda - 1)^T$ is the Perron-Frobenius eigenvector of $A_1$. Since $x$ is the normalized eigenvector of $A$, we have $\sum_{j = 1}^s a_{ij}x_j = \lambda x_i$ and $\sum_{i = 1}^s \sum_{j = 1}^s a_{ij}x_j = \lambda$. Then
$$
\lambda - 1 = \sum_{j = 1}^s x_j \left(\sum_{i = 1}^s  a_{ij}  - 1\right) \in H(x_1,...,x_s).
$$
 Denote by $a_{ij}^{(k)}$ the entries of the matrix $A^k$. We have
 $$
 \lambda^k x_i = \sum_{j = 1}^s a_{ij}^{(k)}x_j = \sum_{j = 1}^s \left(a_{ij}^{(k)} - (\sum_{l = 1}^s  a_{lj}  - 1)\right)x_j + \lambda - 1.
 $$
 Clearly, all the coefficients $a_{ij}^{(k)} - (\sum_{l = 1}^s  a_{lj}  - 1)$ are positive integers for sufficiently large $k$. We have $\lambda^k (\lambda - 1) = \sum_{j=1}^s(\sum_{i=1}^s (a_{ij}^{(k+1)} - a_{ij}^{(k)}))x_j$. It is obvious that the numbers $\sum_{i=1}^s (a_{ij}^{(k+1)} - a_{ij}^{(k)})$ are positive integers for all $k \in \mathbb{N}$ and $1 \leq j \leq s$. Hence we define
$$
A_1 =
\begin{pmatrix}
a_{11}^{(k)} - (\sum_{i = 1}^s  a_{i1}  - 1) & \ldots & a_{1n}^{(k)} - (\sum_{i = 1}^s  a_{i,n}  - 1) & 1\\
\\
\vdots & \ddots & \vdots & \vdots\\
\\
a_{n1}^{(k)} - (\sum_{i = 1}^s  a_{i1}  - 1) & \ldots & a_{nn}^{(k)} - (\sum_{i = 1}^s  a_{i,n}  - 1) & 1\\
\\
\sum_{i=1}^s (a_{i1}^{(k+1)} - a_{i1}^{(k)})& \ldots & \sum_{i=1}^s (a_{i,n}^{(k+1)} - a_{i,n}^{(k)}) & 0
\end{pmatrix}.
$$
It is straightforward to check that $A_1 y = \lambda^k y$.
Since $\lambda - 1 \in H(x_1,...,x_s)$, we have $H(y_1,...,y_{s+1}) = H(x_1,...,x_s)$. The normalized Perron-Frobenius eigenvector of $A_1$ is $z = \frac{1}{\lambda} y$. Since $\lambda H(x) \subset H(x)$, we see that $S(A) = S(A_1)$.

To complete the proof, we note that the construction of $A_{n+1}$ uses $A_n$ in the same way as the construction of $A_1$ uses $A$.
\end{proof}

\begin{theorem}\label{diff_alphabets}
Let $\sigma$ be a proper substitution. Then there exist countably many proper substitutions $\{\zeta_n\}_{n=1}^\infty$ such that $(X_\sigma, T_\sigma)$ is orbit equivalent to $(X_{\zeta_n}, T_{\zeta_n})$, but the systems $\{(X_{\zeta_n}, T_{\zeta_n})\}_{n=1}^\infty$ are pairwise non-isomorphic.
\end{theorem}

\begin{proof}  Let $A$ be the incidence matrix of substitution $\sigma$ defined on an alphabet $\{a_1,..., a_s\}$. Let $\lambda$ be the Perron-Frobenius eigenvalue of $A$ and $x = (x_1,...,x_s)^T$ be the normalized Perron-Frobenius eigenvector.
By Theorem~\ref{complexity_function}, there exists $C > 0$ such that $p_\sigma(n) \leq Cn$ for every $n \geq 1$. We can assume $C \in \mathbb{N}$ and $C+1 > s$. By Lemma~\ref{matrixbuild}, there exists a primitive matrix $\widetilde{A} \in Mat(\mathbb{N},C + 2)$ such that $S(\widetilde{A}) = S(A)$.
We define a primitive substitution $\zeta$ on the alphabet $\{a_1,...,a_{C+2}\}$ such that $A_\zeta = \widetilde{A}$. We can always assume that $(\widetilde{A})_{1i} \geq 2$, for $i = 1,...,C+2$, otherwise we would take the power of $\widetilde{A}$ instead of $\widetilde{A}$.
We require also that the word $\zeta(a_j)$ starts with the letters $a_1 a_j$ and ends with the letter $a_1$ for all $j = 1,...,C+2$. Then $\zeta$ is a proper primitive substitution.

Denote by $u = \lim_{n \rightarrow \infty} \zeta^n(a_1.a_1)$ the unique fixed point for $\zeta$.
Since $\zeta$ is primitive, the sequence $u$ contains $\zeta^k(a_j) = \zeta^{k-1}(a_1)\zeta^{k-2}(a_1)...\zeta(a_1)a_1 a_j...a_1$ as a subword for $k \geq 1$ and $j = 1,..,C+2$. We have $p_u(1) = C + 2$. Since all letters $\{a_j\}_{j=1}^{C+2}$ can follow the word $\zeta^{k-1}(a_1)\zeta^{k-2}(a_1)...\zeta(a_1)a_1$ for any $k$, we have $p_u(n+1) - p_u(n) \geq C + 1$ for $n \geq 1$ by (\ref{fund_equal}). It follows that $p_u(n) > (C+1)n$ for all $n \geq 1$. Hence, by Theorem~\ref{complexity_function}, the systems $(X_\sigma, T_\sigma)$ and $(X_{\zeta}, T_\zeta)$ are not isomorphic.

We will apply induction to produce a needed sequence of substitutions. For $\zeta_1 := \zeta$, there exists $C_1 > 0$ such that $p_{\zeta_1}(n) \leq C_1 n$ for every $n \geq 1$. By the same method as above, we can construct $\zeta_{m+1}$ using $\zeta_m$ for $m \in \mathbb{N}$ such that $S(A_{\zeta_m}) = S(A)$. We obtain $C_m n \leq p_{\zeta_{m+1}}(n) \leq C_{m+1} n$ for $m \in \mathbb{N}$ and $n \geq 1$. Hence, by Theorem~\ref{complexity_function} and Lemma \ref{matrixbuild}, we obtain countably many pairwise non-isomorphic substitution dynamical systems $\{(X_{\zeta_m}, T_{\zeta_m})\}_{m=1}^\infty$ in the orbit equivalence class of $(X_\sigma, T_\sigma)$.
\end{proof}

Now, given any primitive substitution, we find a stationary simple properly ordered Bratteli diagram with the least possible number of vertices such that the corresponding dynamical systems are orbit equivalent. This result is, in some sense, relative to Theorem~\ref{oeodometer}.
We recall some notions and results from~\cite{S.B.O.K.}.

Let $A \in Mat(\mathbb{N},s)$ be a primitive matrix. Let $\lambda$ be a Perron-Frobenius eigenvalue of $A$ and $x = (x_1,...,x_s)^T$ be the normalized Perron-Frobenius eigenvector. Denote by $k$ the degree of the algebraic integer $\lambda$. The number field $\mathbb{Q}(\lambda)$ is a subfield of $\mathbb{R}$ whose elements are written down as $\{a_0 + a_1 \lambda + ... + a_{k-1}\lambda^{k-1} : a_0,a_1,...,a_{k-1} \in \mathbb{Q}\}$.  The numbers $1, \lambda, ... ,\lambda^{k-1}$ form a basis of $\mathbb{Q}(\lambda)$ as a vector space over $\mathbb{Q}$.
If we need to emphasize that a real number $y = a_0 + a_1 \lambda + \ldots + a_{k-1} \lambda^{k-1} \in \mathbb{Q}(\lambda)$ is considered as a vector $(a_0, a_1,...,a_{k-1})^T \in  \mathbb{Q}^k$, we will use the notation $\textbf y$. Let $\textbf{n}$ denote the vector $(1, \lambda, ..., \lambda^{k-1})^T \in \mathbb{R}^k$. Then, for any $\textbf{y}\in \mathbb Q^k$, the corresponding number $y \in \mathbb{R}$ can be written as $y = \langle \textbf{y}, \textbf{n}\rangle$.

Let $B$ be a stationary simple Bratteli diagram with incidence matrix $F = A^T$ and no multiple edges between levels 0 and 1. Let $\mu$ be its unique ergodic probability $\mathcal{R}$-invariant measure. Let $G(S(\mu))$ be an additive subgroup of reals generated by $S(\mu)$. It is not hard to see that $S(A) = G(S(\mu)) \subset \mathbb{Q}(\lambda)$.

Let $p(\lambda)$ be the polynomial in $\mathbb Q(\lambda)$ such that $\lambda^{-1}= p(\lambda)$. The map $y \mapsto p(\lambda)y$ in $\mathbb Q(\lambda)$ determines a linear transformation in the vector space $\mathbb{Q}^k$. Let $D$ be the matrix which corresponds to this transformation.
The following results can be found in the proof of Theorem 3.2 in~\cite{S.B.O.K.}.
The matrix $D \in Mat(\mathbb{Q}, k)$ is a nonsingular matrix and for any $\textbf{x} \in \mathbb{Q}^k$ the vector $D\textbf{x} \in \mathbb{Q}^k$ corresponds to the number $\frac{x}{\lambda} \in \mathbb{R}$.
Let $C = D^{-1}$. The matrix $C$ has only one eigenvector $\textbf{y}_1$ with eigenvalue $\lambda$, the absolute value of any other eigenvalue of $C$ is less than $\lambda$, and $\textbf{y}_1$ is the only eigenvector of $C$ that is not orthogonal to $\textbf{n}$. Obviously, we can assume $\langle\textbf{y}_1,\textbf{n}\rangle > 0$. Denote by $\pi = \{\textbf y \in \mathbb{R}^k: \langle \textbf{y}, \textbf{n}\rangle = 0\}$. The iterations of $C$ drive any ray which is not in $\pi$ to the limit ray generated by $\textbf{y}_1$, the iterations of $D = C^{-1}$ do the opposite thing. More precisely, if $\langle \textbf{y}, \textbf{n}\rangle \neq 0$ for some $\textbf{y} \in \mathbb{Q}^k$ then angle between the line generated by $D^N \textbf{y}$ and $\pi$ can be made arbitrary small when $N$ tends to infinity.

\begin{remark}\label{nec_cond_lambda}
For $A \in Mat(\mathbb{N},s)$, let $\lambda$ be the Perron-Frobenius eigenvalue
for $A$ and $x = (x_1,....,x_s)^T$ be the corresponding eigenvector. Then $\deg
\lambda = \deg\lambda^m$ for all $m \in \mathbb{N}$. Indeed, since the Perron-Frobenius eigenvalue of $A^m$ is $\lambda^m$ and $A^m x =\lambda^m x$, we have $\deg \lambda^m = \dim(Lin_{\mathbb{Q}}\{\textbf{x}_1,...,\textbf{x}_k\}) = \deg \lambda$.
\end{remark}

\begin{lemma}\label{min_matrix}
Let $A \in Mat(\mathbb{N},s)$ be a primitive matrix. Let $\lambda$ be the Perron-Frobenius eigenvalue of $A$. Let $k \geq 2$ be the degree of algebraic number $\lambda$. Then there exists a primitive matrix $\widetilde{A} \in Mat(\mathbb{N},k)$ such that $S(\widetilde{A}) = \alpha S(A)$ for some positive $\alpha \in \mathbb{Q}(\lambda)$. Moreover, $k$ is the least possible dimension for which this equality holds.
\end{lemma}

\begin{proof}  Let $x = (x_1,...,x_s)^T$ be the normalized Perron-Frobenius eigenvector of $A$.
Let $\Lambda(\textbf{x}_1,...,\textbf{x}_s)$ be the lattice in $\mathbb R^k$  generated by $\textbf{x}_1,...,\textbf{x}_s$. Then $\Lambda(\textbf{x}_1,...,\textbf{x}_s) \subset \mathbb{Q}^k$ corresponds to the group $H(x_1,...,x_s) \subset \mathbb{R}$.
There exist $\textbf{f}_1,...,\textbf{f}_k \in \mathbb{Q}^k$ such that $\Lambda(\textbf{f}_1,...,\textbf{f}_k) = \Lambda(\textbf{x}_1,...,\textbf{x}_s)$ (see~\cite{Nath}). Since $1, \lambda,...,\lambda^{k-1} \in H(x)$, the vectors $\{\textbf{f}_i\}_{i=1}^k$ form a basis of $\mathbb{Q}^k$.
By changing $\textbf{f}_i$ to $-\textbf{f}_i$, we can make all vectors $\{\textbf{f}_i\}_{i=1}^k$ satisfy the inequality $\langle \textbf{f}_i, \textbf{n}\rangle > 0$. Since $D$ is a nonsingular linear transformation of $\mathbb{Q}^k$, the vectors $\{D^N \textbf{f}_i\}_{i=1}^k$ form a basis of $\mathbb{Q}^k$ for every $N \in \mathbb{N}$ and $\Lambda(D^N\textbf{f}_1,...,D^N\textbf{f}_k) = \Lambda(D^N\textbf{x}_1,...,D^N\textbf{x}_s)$.
Consider the cone
$$
K(D^N \textbf{f}_1,...,D^N \textbf{f}_k) = \{\sum_{i=1}^k \beta_i D^N \textbf{f}_i : \beta_i \geq 0, i = 1,...,k\}.
$$
 Recall that $\textbf{y}_1$ is an eigenvector of $C$ such that $\langle\textbf{y}_1, \textbf{n}\rangle > 0$ and the iterations of $C$ drive any ray which is not in $\pi$ to the limit ray generated by $\textbf{y}_1$. There exists $N \in \mathbb{N}$ such that the vectors $\textbf{y}_1, \textbf{x}_1,...,\textbf{x}_s$ lie in the cone $K(D^N \textbf{f}_1,...,D^N \textbf{f}_k)$. Then there exists an integer $M > 0$ such that $D^{N-M} \textbf{f}_i \in K(D^N\textbf{f}_1,...,D^N\textbf{f}_k)$ for $i = 1,...,k$. Since $\lambda H(x) \subset H(x)$, we also have $D^{N-M} \textbf{f}_i \in \Lambda(D^N\textbf{f}_1,...,D^N\textbf{f}_k)$ for $i = 1,...,k$.
Since $\textbf{f}_1,...,\textbf{f}_k$ are linearly independent, there exist positive integers $\{\tilde{a}_{ij}\}_{i,j = 1}^k$ such that $D^{N-M} \textbf{f}_i = \sum_{j=1}^k \tilde{a}_{ij} D^N \textbf{f}_j$ for $i = 1,...,k$. Set $z_i = \langle D^N \textbf{f}_i, \textbf{n}\rangle$ and $\widetilde{A} = (\tilde{a}_{ij})_{i,j = 1}^k$.
Then
$$
z = \left(\frac{z_1}{\sum_{l=1}^k z_k},...,\frac{z_k}{\sum_{l=1}^k z_k}\right)^T \in \mathbb{R}^{k}
$$
is a normalized Perron-Frobenius eigenvector for $\widetilde{A}$ with eigenvalue $\lambda^M$.
Setting $\alpha = \sum_{l=1}^k z_k$, we obtain that $\alpha S(A) = S(\widetilde{A})$ because $H(z_1,...,z_k) = \frac{1}{\lambda^N}H(x_1,...,x_s)$.

Now we show that if $P \in Mat(\mathbb{N},l)$ such that $S(P) = \alpha S(A)$ for some $\alpha \in \mathbb{R}$ then $l \geq k$.
First, we show that $\deg \beta = d \geq k$. Assume that the converse holds. Suppose $\deg \beta = d < k$. Recall that $1, \lambda, ... , \lambda^{k-1} \in H(x_1,...,x_s)$. We have $S(P) \subset \mathbb{Q}(\beta)$. Then the elements of $S(P) = \alpha S(A)$ can be represented as some vectors of $\mathbb{Q}^d$. In particular, $\alpha, \alpha \lambda, ...,\alpha \lambda^{k-1}$ can be represented as $\mathbf v_1,...,\mathbf v_k \in \mathbb{Q}^d$. Since $d < k$, the vectors $\{\mathbf v_i\}_{i=1}^k$ are linearly dependent over $\mathbb{Q}$, hence there exist rational numbers $\{r_i\}_{i = 1}^k$ such that $\sum_{i=1}^k r_i \mathbf v_i = 0$. Returning from $\mathbb{Q}^d$ to $\mathbb{R}$ we obtain $\alpha \sum_{i=1}^k r_i \lambda^{i} = 0$. But then the algebraic degree of $\lambda$ is less than $k$. This is a contradiction, hence $d \geq k$. Since $\beta$ is a root of characteristic polynomial for $P$, the dimension of $P$ is not less than $k$.
\end{proof}

\begin{remark}
Given a Perron number $\lambda$ with $\deg \lambda = k$, we find a primitive matrix  $\widetilde{A} \in Mat(\mathbb{N},k)$ such that $\lambda^M$ is the Perron-Frobenius eigenvalue of $\widetilde{A}$ for some $M \in \mathbb{N}$. In~\cite{Lind}, it was shown that there may not exist a matrix $\widetilde{A} \in Mat(\mathbb{N},k)$ with the Perron-Frobenius eigenvalue $\lambda$ (see Example~\ref{PerronLind} below).
\end{remark}

Recall that a stationary Bratteli diagram  may have multiple edges between levels 0 and 1. The following theorem is a generalization of Theorem~\ref{oeodometer}.

\begin{theorem}\label{minBD}
Let $\sigma$ be a primitive substitution whose incidence matrix has a Perron-Frobenius eigenvalue $\lambda$ and $k = \deg \lambda$. Then $(X_\sigma, T_\sigma)$ is orbit equivalent to a Bratteli-Vershik system defined on a stationary Bratteli diagram with $k$ vertices on each level. Moreover, there is no stationary Bratteli-Vershik system  with less than $k$ vertices which is orbit equivalent to $(X_\sigma, T_\sigma)$.
\end{theorem}

\begin{proof}  Suppose $\sigma$ is not a proper substitution. Then, by Theorem~\ref{return_words}, there exists a proper substitution $\zeta$ with incidence matrix $A_\zeta$ such that $(X_\zeta,T_\zeta)$ is isomorphic to $(X_\sigma, T_\sigma)$ and the Perron-Frobenius eigenvalue of $A_\zeta$ is $\lambda^d$ for some $d \in \mathbb{N}$. By Remark~\ref{nec_cond_lambda}, we have $\deg \lambda = \deg \lambda^d = k$.
Thus, without loss of generality, we may assume that $\sigma$ is a proper substitution.

We will use the notation from Lemma~\ref{min_matrix}. Let $A$ be the incidence matrix for $\sigma$. There exist a primitive matrix $\widetilde{A} \in Mat(\mathbb{N},k)$ and a positive number $\alpha \in \mathbb{Q}(\lambda)$ such that $S(\widetilde{A}) = \alpha S(A)$. Let $(B, \leq)$ be a stationary ordered Bratteli diagram corresponding to $\sigma$ and $\mu$ be the unique invariant measure for the Bratteli-Vershik system $(X_B, \phi_B)$.
Suppose  $\widetilde{B}$ is the stationary Bratteli diagram with incidence matrix $\widetilde{F} = \widetilde{A}^T$ and no multiple edges between levels 0 and 1.
Let $x = (x_1,...,x_s)^T$ be a normalized Perron-Frobenius eigenvector for $A$. The diagram $\widetilde{B}$ has $k$ vertices and $\frac{1}{\alpha}(\langle D^N \textbf{f}_1, \textbf{n}\rangle,...,\langle D^N \textbf{f}_k, \textbf{n}\rangle)^T$ is the  normalized Perron-Frobenius eigenvector for $\widetilde{A}$.
Since the vectors $\textbf{x}_1,...,\textbf{x}_s$ lie in the positive cone $K(D^N \textbf{f}_1,...,D^N \textbf{f}_k)$ and in the lattice $\Lambda(D^N\textbf{f}_1,...,D^N\textbf{f}_k)$, each $\textbf{x}_i$ is a linear combination of $\{D^N\textbf{f}_j\}_{j=1}^k$ with natural coefficients.
We make a finite change between the zero and first levels of $\widetilde{B}$ and obtain $B_1$ as follows. If $\textbf{x}_i = \sum_{j=1}^k b_{ij}D^N\textbf{f}_j$ then let $B_1$ have $\sum_{i=1}^s b_{ij}$ edges between $v_0$ and $j$-th vertex of the first level. Let $\leq_1$ be a proper order on $B_1$. Let $\nu$ be the unique invariant measure for $(X_{B_1}, \phi_{B_1})$. Then $(B_1,\leq_1)$ is a stationary Bratteli diagram such that $G(S(\nu)) = G(S(\mu))$. Hence the Bratteli-Vershik system on $(B_1,\leq_1)$ is orbit equivalent to $(X_\sigma, T_\sigma)$. Let $\widetilde{\leq}$ be any proper order on $\widetilde{B}$.
Note that the proper substitution systems associated to $(B,\leq)$ and $(\widetilde{B},\widetilde{\leq})$ are Kakutani orbit equivalent.
By Lemma~\ref{min_matrix}, there is no stationary Bratteli-Vershik system on the diagram with less than $k$ vertices which is orbit equivalent to $(X_\sigma, T_\sigma)$.
\end{proof}

\begin{remark}\label{mult_edges}
The diagram $(B_1, \leq_1)$ has multiple edges between levels 0 and 1. By Theorem~\ref{mult_edges_thm}, there exists a stationary Bratteli diagram $(B_2, \leq_2)$ such that $B_2$ has no multiple edges between levels 0 and 1 and Bratteli-Vershik systems $(X_{B_1}, \phi_{B_1})$, $(X_{B_2}, \phi_{B_2})$ are isomorphic. Let $A_2$ be the matrix transpose to the incidence matrix of $B_2$. Then $S(A_2) = S(A)$.
\end{remark}

The corollary easily follows from the proof of Theorem~\ref{minBD} and Remark~\ref{mult_edges}.

A \textit{Perron number} is a real algebraic integer greater than one,  that is larger than the absolute value of any of its Galois conjugates \cite{Lind}.

\begin{corol}
Let $\lambda \in \mathbb{R}$ be a Perron number and  $\{x_i\}_{i=1}^s \subset \mathbb{Q}(\lambda) \cap (0, \infty)$ with $\sum_{i=1}^s x_i = 1$. Let $H =  H(x_1,...,x_s)$ be the additive group generated by $x_1,...,x_s$. Suppose  $\{\textbf{x}_i\}_{i=1}^s$ are vectors in $\mathbb{Q}^k$ corresponding to $\{x_i\}_{i=1}^s$ and $Lin_{\mathbb{Q}} \{\textbf{x}_i\}_{i=1}^s$ denotes the set of all rational linear combinations of $\{\textbf{x}_i\}_{i=1}^s$.

(i) If $Lin_{\mathbb{Q}} \{\textbf{x}_i\}_{i=1}^s = \mathbb{Q}^k$ and  $\lambda^M H \subset H$ for some  $M \in \mathbb{N}$, then there exists a primitive matrix $A$ with natural entries such that $S(A) = \bigcup_{N=0}^\infty \frac{1}{\lambda^N} H(x_1,...,x_s)$.

(ii) If $\lambda H \subset H$, then there exists a primitive matrix $A$ with natural entries such that $S(A) = \bigcup_{N=0}^\infty \frac{1}{\lambda^N} H(x_1,...,x_s)$.
\end{corol}

\begin{prop}
Let $\lambda \in \mathbb{Q}$ and $\{x_i\}_{i=1}^s \subset \mathbb{Q} \cap (0, \infty)$ with $\sum_{i=1}^s x_i = 1$. Let $H(x) =  H(x_1,...,x_s)$ be an additive group generated by $x_1,...,x_s$.
Let $Y = \bigcup_{l \in \mathbb{N}}\{\{y_j\}_{j=1}^l \subset \mathbb{Q}\cap (0, \infty): \sum_{j=1}^l y_j = 1 \mbox{ and } H(y) = H(x)\}$. Then $Y$ is a finite set.
\end{prop}

\begin{proof}
Let $x_i = \frac{p_i}{q}$, where $\gcd(p_1,...,p_s) = 1$. We have $H(x) = \frac{1}{q}\mathbb{Z}$. Then $Y = \bigcup_{l \in \mathbb{N}}\{\{y_j\}_{j=1}^l \subset \mathbb{Q}\cap [0, \infty): \sum_{j=1}^l y_j = 1 \mbox{ and } H(y) = \frac{1}{q}\mathbb{Z}\}$. Hence $y_i = \frac{q_i}{q}$ for some $q_i \in \mathbb{N}$ such that $\sum_{i=1}^l q_i = q$. Since the number of partitions of $q$ into natural numbers is finite, we obtain $|Y| < \infty$.
\end{proof}

More results, which are related to the two statements above, can be found in \cite{Lind}.

\begin{example}
Here is an example illustrating Theorem~\ref{diff_alphabets}. Let $B_0$ be a stationary  simple Bratteli diagram with the matrix $A_0$ transpose to the incidence matrix:

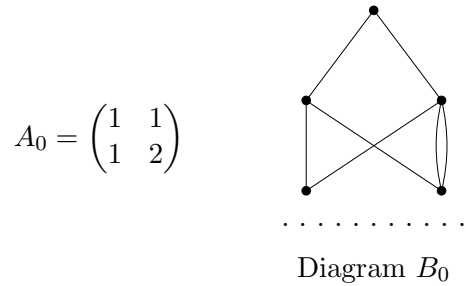
\begin{center}
\begin{tabular*}{0.5\textwidth}%
     {@{\extracolsep{\fill}}cc}
$
A_0 =
\begin{pmatrix}
1 & 1\\
1 & 2
\end{pmatrix}
$
            &
\unitlength = 0.3cm
\begin{graph}(8,7)
\graphnodesize{0.4}
\roundnode{V0}(4,6)
\roundnode{V10}(1,2)
\roundnode{V11}(7,2)
\edge{V10}{V0}
\edge{V11}{V0}

\roundnode{V20}(1,-2)
\roundnode{V21}(7,-2)
\edge{V20}{V10}
\edge{V20}{V11}
\edge{V21}{V10}
\bow{V21}{V11}{0.06}
\bow{V21}{V11}{-0.06}

\freetext(4,-3.5){$.\ .\ .\ .\ .\ .\ .\ .\ .\ .\ .$}%
\freetext(4,-5.5){{Diagram $B_0$}}
\end{graph}

\vspace{1.5cm}
\end{tabular*}
\end{center}

Let $\leq_0$ be a proper order for $B_0$ and $\sigma$ be a substitution read on $(B_0, \leq_0)$. For example, choose
$$
\sigma=\left\{\begin{array}{l} a \mapsto ab\\b \mapsto abb\\ \end{array}\right.
$$
Since substitution $\sigma$ is Sturmian, we have $p_\sigma(n) = n + 1$ (see~\cite{Fogg}). By Lemma~\ref{matrixbuild}, we build a primitive matrix $A_1 \in Mat(\mathbb{N},3)$ such that $S(A_1) = S(A)$:

$$
A_1 =
\begin{pmatrix}
1 & 1 & 1\\
2 & 3 & 1\\
8 & 13 & 0
\end{pmatrix}.
$$

Consider a proper order $\leq_1$ on $B_1$ with the following substitution $\zeta$ read on $(B_1, \leq_1)$:
$$
\zeta=\left\{\begin{array}{l} a \mapsto abbcccccccc\\b \mapsto abbbccccccccccccc\\ c \mapsto ab \end{array}\right.
$$
It can be proved that $p_\zeta(n) \geq 3n$. Hence we already obtain that $(X_\sigma, T_\sigma)$ and $(X_\zeta, T_\zeta)$ are non-isomorphic orbit equivalent systems.

\begin{remark}
Let $\lambda$ be the Perron-Frobenius eigenvalue of $A_0$ and $x$ be the Perron-Frobenius eigenvector. Let $\mu$ be the unique $\mathcal{R}$-invariant measure for the corresponding diagram. Then $\lambda H(x) = H(x)$ and $G(S(\mu)) = H(x)$ (see \cite{S.B.O.K.}).
\end{remark}
\end{example}

\begin{example}\label{PerronLind}
The following example concerns Lemma~\ref{min_matrix}. Lind~\cite{Lind} pointed out  an example of Perron number $\lambda$ ($\lambda \approx 3.8916$ is the Perron root of the equation $f(t) = t^3 + 3t^2 - 15t -46$ and $\deg \lambda = 3$) such that  there exist no matrix $A \in Mat(\mathbb{N},3)$ with  Perron-Frobenius eigenvalue $\lambda$. The reason is that a $3$-dimensional matrix with spectral radius $\lambda$ has trace $- 3$, and hence cannot be non-negative. Lemma~\ref{min_matrix} states that such a matrix must exist for some power of $\lambda$. We present here a matrix $\widetilde{A} \in Mat(\mathbb{N},3)$ with  Perron-Frobenius eigenvalue $\lambda^M$ for some $M \in \mathbb{N}$.  In notation used in Lemma \ref{min_matrix}, we notice that
$$
C =
\begin{pmatrix}
0 & 0 & 46\\
1 & 0 & 15\\
0 & 1 & - 3\\
\end{pmatrix},
$$
and $\textbf{y}_1 = (1, \frac{1}{\lambda} + \frac{15}{46}, \frac{1}{\lambda^2} + \frac{15}{46} \frac{1}{\lambda} - \frac{3}{46})^T$ has positive coordinates (see also ~\cite{S.B.O.K.}). Let $\{\textbf{e}_i\}_{i=1}^3$ be the standard basis of $\mathbb{Q}^3$, hence $\langle \textbf{e}_i, \textbf{n}\rangle = \lambda^{i-1} > 0$. Then the iterates of matrix $C$ drive each $\textbf{e}_i$ closer to $\textbf{y}_1$ when $M$ is growing, hence  $C^M \textbf{e}_i$ has all positive coordinates for sufficiently large $M$. We can choose $\widetilde{A} = C^M$. In this specific example, it suffices to take  $M \geq 49$.
\end{example}

\bibliographystyle{amsplain}

\end{document}